\documentclass[11pt,a4paper]{amsart}
\usepackage{amssymb,amsmath,epsfig,graphics,mathrsfs,enumerate,verbatim}
\usepackage[pagebackref,colorlinks=true,linkcolor=blue,citecolor=blue]{hyperref}
\usepackage{fancyhdr}
\usepackage{hyperref}
\hypersetup{
 colorlinks   = true,
 urlcolor     = blue,
 linkcolor    = blue,
 citecolor   = red ,
 bookmarksopen=true
}

\usepackage{amsmath}
\usepackage{amsfonts}
\usepackage{amssymb}
\usepackage{amsthm}
\usepackage{epsfig,graphics,mathrsfs}
\usepackage{graphicx}
\usepackage{dsfont}

\usepackage[usenames, dvipsnames]{color}

\usepackage{hyperref}

\def \phi {\varphi}

\def \RN {\mathbb{R}^N}
\def \R {\mathbb{R}}

\def \G{\Gamma}
\newcommand{\Ba}{\mathscr B_\alpha}

\newcommand{\Rn}{\mathbb R^n}
\newcommand{\Rm}{\mathbb R^m}
\newcommand{\Om}{\Omega}

\newcommand{\p}{\partial}

\newcommand{\la}{\lambda}

\newcommand{\Za}{Z_\alpha}
\newcommand{\ra}{\rho_\alpha}
\newcommand{\na}{\nabla_\alpha}

\numberwithin{equation}{section}

\newcommand{\beq}{\begin{equation}}
\newcommand{\bea}[1]{\begin{array}{#1} }
\newcommand{\eeq}{ \end{equation}}
\newcommand{\ea}{ \end{array}}

\newcommand{\sul}{\Delta_H}

\newcommand{\sa}{\langle}
\newcommand{\da}{\rangle}

\newcommand{\quo}{\frac{d\sigma}{|\nabla \rho|}}
\newcommand{\qui}{\psi \frac{d\sigma}{|\nabla \rho|}}

\newtheorem{theorem}{Theorem}[section]
\newtheorem{lemma}[theorem]{Lemma}
\newtheorem{proposition}[theorem]{Proposition}

\newtheorem{remark}[theorem]{Remark}

\numberwithin{equation}{section}

\begin{document}

\title[A Rellich type estimate,  etc.]{A Rellich type estimate for a subelliptic Helmholtz equation with mixed homogeneities}

\keywords{Helmholtz equation, Sharp decay of eigenfunctions, Liouville type theorem}

\keywords{Baouendi-Grushin operators. Eigenfunctions. Rellich type estimates}
\subjclass{35P15, 35Q40, 47A75}

\date{}

\begin{abstract}
We establish an optimal asymptotic on rings for a subelliptic Helmholtz equation with mixed homogeneities. 
\end{abstract}

\author{Agnid Banerjee}
\address{Tata Institute of Fundamental Research\\
Centre For Applicable Mathematics \\ Bangalore-560065, India}\email[Agnid Banerjee]{agnidban@gmail.com}

\author{Nicola Garofalo}

\address{Dipartimento d'Ingegneria Civile e Ambientale (DICEA)\\ Universit\`a di Padova\\ Via Marzolo, 9 - 35131 Padova,  Italy}
\vskip 0.2in
\email{nicola.garofalo@unipd.it}

\thanks{A. Banerjee  is supported in part  by Department of Atomic Energy,  Government of India, under
project no.  12-R \& D-TFR-5.01-0520. N. Garofalo is supported in part by a Progetto SID (Investimento Strategico di Dipartimento): ``Aspects of nonlocal operators via fine properties of heat kernels", University of Padova (2022); and by a PRIN (Progetto di Ricerca di Rilevante Interesse Nazionale) (2022): ``Variational and analytical aspects of geometric PDEs". He has also been partially supported by a Visiting Professorship at the Arizona State University. }

\maketitle


\section{Introduction}\label{S:intro}

Understanding the spectrum of the Laplacian is a classical problem in analysis and geometry. A seminal contribution to the subject was given by F. Rellich in \cite{Rel}, where he proved the following remarkable estimate. 
Suppose that $f\not\equiv 0$ solves the Helmholtz equation 
\begin{equation}\label{helmo0}
\Delta f = - \kappa^2\ f,\ \ \ \ \ \ \ \ \ \ \ \ \kappa>0,
\end{equation}
in the complement of a cavity $\Om_0 = \{x\in \Rn\mid |x|>R_0\}$, for some $R_0>0$. Then there exist $R_1>R_0$ sufficiently large, and a constant $M = M(n,\kappa,f)>0$, such that for every $R>R_1$ one has
\begin{equation}\label{rellich}
\int_{R<|x|<2R} |f|^2 dx \ge M R.
\end{equation} 
As it is well-known, an immediate consequence of \eqref{rellich} is that there exist no solutions of \eqref{helmo0} in $L^2(\Om_0)$. Important extensions of these results to eigenfunctions of general Schr\"odinger operators were found by Kato \cite{Ka}, Agmon \cite{A} and Simon \cite{RS}. 
 
One situation which is of considerable interest and to a large extent unexplored is understanding the spectrum of differential operators on product spaces with mixed homogeneity. A basic example is the operator in $\R^N$, with $N=m+k$, 
\begin{equation}\label{pbeta}
\mathscr B_{\alpha} = \Delta_z  + \frac{|z|^{2\alpha}}4 \Delta_\sigma,
\end{equation}
where $z\in \Rm$, $\sigma\in \R^k$, and $\alpha>0$. This operator is invariant with respect to the anisotropic dilations \eqref{dilB} below. When $\alpha\to 0$, up to a renormalising factor, $\Ba$ converges to the standard Laplacian in $\RN$, but for $\alpha>0$ its ellipticity degenerates on the $k$-dimensional subspace $M = \{0\}\times \R^k\subset \R^N$. 
The Dirichlet problem for partial differential equations such as  \eqref{pbeta} was first studied by S. Baouendi in his doctoral dissertation \cite{B}. Subsequently, Grushin and Vishik  studied hypoellipticity questions in \cite{Gr1}, \cite{Gr2}, \cite{GV1}, \cite{GV2}.

When $\alpha = 1$ the operator \eqref{pbeta} is relevant in complex hyperbolic geometry as it is connected to the horizontal Laplacian $\sul$ in a group of Heisenberg type, see \cite{Fo},  \cite{Kap} and \cite{Be}. In such framework, in fact, in the logarithmic coordinates $(z,\sigma)$ with respect to a fixed orthonormal basis of the Lie algebra one has
\begin{equation}\label{slH}
\Delta_H = \Delta_z + \frac{|z|^2}{4} \Delta_\sigma  + \sum_{\ell = 1}^k \p_{\sigma_\ell} \sum_{i<j} b^\ell_{ij} (z_i \p_{z_j} -  z_j \p_{z_i}),
\end{equation}
where $b^\ell_{ij}$ indicate the group constants. When for instance $f(z,\sigma) = g(|z|,\sigma)$ is a function with cylindrical symmetry, then by the skew-symmetry of $b^\ell_{ij} $ one has $\sum_{i<j} b^\ell_{ij} (z_i \p_{z_j} -  z_j \p_{z_i}) f = 0$, and therefore
\[
\sul f = \Delta_z f+ \frac{|z|^2}4 \Delta_\sigma f,
\] 
which is precisely \eqref{pbeta} when $\alpha=1$. In a different framework, \eqref{pbeta} shows up in the theory of free boundaries, as the extension operator in the famous Caffarelli-Silvestre procedure for the fractional Laplacian $(-\Delta)^s$, $0<s<1$, see \cite[formula (1.8) on p.1247]{CaS}, \cite[Theorem 3.1 (see Remark 3.2)]{CSS}, and also \cite[Proposition 11.2]{Gft}. If, in fact, we let $\alpha = \frac{1}{2s} -1\in (-\frac 12,\infty)$, and for $u(\sigma)\in \mathscr S(\Rn)$ consider the solution $U(z,\sigma)$ to the Dirichlet problem
\begin{equation}\label{bg00}
\begin{cases}
\p_{zz} U(z,\sigma) + z^{2\alpha}  \Delta_\sigma U(z,\sigma)  = 0,\ \ \ \ (z,\sigma)\in \R^{n+1}_+ = \R_z^+\times \Rn_\sigma,
\\
U(0,\sigma) = u(\sigma),
\end{cases}
\end{equation}
then one has
\begin{equation}\label{dtnbg}
(-\Delta)^s u(\sigma) = - \frac{\G(1+s)}{\G(1-s)}\ \underset{z\to 0^+}{\lim} \p_z U(z,\sigma).
\end{equation}

Besides these connections, operators with mixed homogeneities such as $\Ba$ are of interest in several different contexts in analysis and geometry, and there is a large and continuously growing literature which would of course be impossible to list in this note. We confine ourselves to mention the works \cite{RS}, \cite{Gre}, \cite{Je}, \cite{FLto}, \cite{FL}, \cite{FL84}, \cite{FL85}, \cite{FS}, \cite{Gjde}, 
 \cite{GS}, \cite{BGG1}, \cite{BGG2}, \cite{Be}, \cite{MM},  \cite{BFIh}, \cite{BFI}, \cite{GR},  \cite{GTpotan}.

One challenging problem directly connected to the present work is understanding the nature of the eigenfunctions of \eqref{pbeta}. Because of the anisotropic behaviour of such operator, this question is largely not understood, and poses subtle new challenges with respect to the standard Laplacian $\Delta$. For the latter, we have recently observed in \cite{BG} that Rellich's $L^2$-estimate \eqref{rellich} can, in fact, be used to prove that there exist no solutions to \eqref{helmo0} such that, for some $0<p\le \frac{2n}{n-1}$, one has 
\begin{equation}\label{p}
\int_{\Om_0} |f|^p dx < \infty.
\end{equation}
We stress that this result is sharp since, for every $\kappa>0$ and $p> \frac{2n}{n-1}$, the function 
\begin{equation}\label{f}
f(x) = \int_{\mathbb S^{n-1}} e^{i \kappa \sa x,\omega\da} d\sigma(\omega) = c(n,\kappa) |x|^{- \frac{n-2}2} J_{\frac{n-2}{2}}(\kappa |x|),
\end{equation}
where $J_\nu$ indicate the Bessel function of the first kind, is a solution to \eqref{helmo0} in $\Rn$ such that $f\in L^p(\Rn)$.

This leads us to introduce our main results. Using a different approach from Rellich's, in the cited work \cite{GS}, Z. Shen and one of us established the following delicate asymptotic relation for the eigenfunctions $f$ of a class of Schr\"odinger operators modelled on \eqref{pbeta}. To state it, consider the anisotropic gauge $\rho_\alpha(z,\sigma)$ defined by \eqref{ra} below, and denote by $\Om_0 = \{(z,\sigma)\in \RN\mid \rho_\alpha(z,\sigma) > R_0\}$ the relevant cavity. Then for sufficiently large $R>R_0$ one has
\begin{equation}\label{franzo} 
\underset{R\to \infty}{\liminf}\ \int_{R<\rho_\alpha(z,\sigma)<2R} |f|^2 dz d\sigma > 0.
\end{equation}
Since the gauge rings $\{R<\rho_\alpha(z,\sigma)<2R\}$ fill up the whole space $\RN$, it is clear that the estimate \eqref{franzo} does imply that $f\not\in L^2(\Om_0)$. In particular, such result allows to recover in the limit as $\alpha\to 0$ the cited ones by Kato-Agmon-Simon. However, \eqref{franzo} is quantitatively weaker than Rellich's estimate \eqref{rellich} above, and it falls short of providing information about eigenfunctions below the $L^2$ threshold.   

\medskip

The aim of this paper is to establish the following nontrivial improvement of the asymptotic relation \eqref{franzo}. In its statement, we indicate with $Q_\alpha$ the homogeneous dimension associated with the anisotropic dilations \eqref{dilB} of $\Ba$, see \eqref{Qa} below. 

\begin{theorem}\label{sharprel}
Let $Q_\alpha\ge 3$ and let $f\not\equiv 0$ be a solution to 
\begin{equation}\label{kappa}
\Ba f = - \kappa^2\ f\ \ \ \ \ \ \ \text{in}\ \ \Om_0,
\end{equation}
for some $\kappa>0$. There exists $R_1>R_0$ sufficiently large, and a  constant $M = M(m,k,f)>0$ such that for $R>R_1$ one has
\begin{equation}\label{franzosharp}
\int_{R<\rho_\alpha(x)<2R} |f|^2 dx \ge M R.
\end{equation} 
\end{theorem}
From Theorem \ref{sharprel}, by arguing as in the proof of \cite[Theorem 1.1]{BG}, we obtain the following Liouville type theorem which represents a sharpening of the above mentioned $L^2$ result in \cite{GS}.  

\begin{theorem}\label{main1}
Let $Q_\alpha\ge 3$ and let $f$ be a solution to \eqref{kappa} in $\Om_0$ for some $\kappa>0$. If for some $0< p \leq \frac{2Q_\alpha}{Q_\alpha-1}$ the function $f$ satisfies
\begin{equation}\label{lp}
\int_{\Om_0} |f|^p dx < \infty,
\end{equation}
then $f \equiv 0$ in $\Om_0$.
\end{theorem}

\begin{remark}\label{R:Q} 
In our proof of Theorems \ref{sharprel} and \ref{main1} the limitation $Q_\alpha \ge 3$ seems to play an essential role. Since by \eqref{Qa} below we have $Q_\alpha = m + (\alpha+1)k$, it is clear that when $m\ge 2$, it is always true that $Q_\alpha\ge 3$, whatever $k\ge 1$ and $\alpha>0$ are. If instead $m=k=1$, in order to guarantee that $Q_\alpha\ge 3$, our results will be confined to the range $\alpha\ge 1$ in \eqref{pbeta}.
\end{remark}  

The organisation of the present paper is as follows. In Section \ref{S:prelim} we collect some known facts about the operator \eqref{pbeta}. Section \ref{S:ri} is devoted to highlighting some of the essential steps in the paper \cite{GS}. We begin by stating the basic Rellich type identity \eqref{gud}. Such identity led the authors of \cite{GS} to the introduction of the functional $F(\ell,t)$, which we recall in \eqref{Fell}. The monotonicity in Lemma \ref{L:Fmono} below is one of the main results about $F(\ell,t)$. Finally, in Section \ref{S:main} we prove Theorems \ref{sharprel} and \ref{main1}. One of the key novelty of the present work is the new functional $G(w,t)$ in \eqref{Gud} below, and its monotonicity established in Lemma \ref{L:ag}. Such result leads to the growth property in Lemma \ref{L:gsimpr2}, which cascades in the improved Lemma \ref{L:gsimpr3}. With this final tool in hands, the proof of Theorem \ref{sharprel} follows along the lines of that of \eqref{franzo} in \cite{GS}.

\medskip

\noindent \textbf{Acknowledgment:} We thank Matania Ben-Artzi, Eugenia Malinnikova, Yehuda Pinchover and Sundaram Thangavelu for insightful conversations about the subject of this paper.

\vskip 0.2in


\section{Preliminaries}\label{S:prelim}

In this section we collect some known facts from \cite{Gjde} and \cite{GS}. The reader should also see the paper \cite{GR}. We equip $\R^N$ with the following non-isotropic dilations
\begin{align}\label{dilB}
 \delta_\lambda(z,\sigma)&=(\lambda z,\lambda^{\alpha+1}\sigma),\ \ \lambda>0.
\end{align}
A function $u$ is $\delta_\la$-homogeneous (or simply, homogeneous) of degree $\kappa$ if 
\[
u(\delta_\la(z,\sigma)) = \la^\kappa u(z,\sigma),\ \ \ \ \ \ \ \ \ \ \ \la>0.
\]
In the study of the operators \eqref{pbeta} the following pseudo-gauge on $\R^{N}$ introduced in \cite{Gjde} plays a pervasive role:
\begin{align}\label{ra}
 \rho_\alpha(z,\sigma) & = \left(|z|^{2(\alpha+1)}+4(\alpha+1)^2|\sigma|^2\right)^{\frac1{2(\alpha+1)}}.
\end{align}
It is clear from \eqref{dilB} and \eqref{ra} that $\rho_\alpha$ is homogeneous of degree one, i.e., $\rho_\alpha(\delta_\la(x)) = \la \rho_\alpha(x)$, for every $x = (z,\sigma)\in \RN$. The infinitesimal generator of the dilations \eqref{dilB} is given by
\begin{equation}\label{Za}
 Z_\alpha=\sum_{i=1}^m z_i\partial_{z_i} + (\alpha+1)\sum_{j=1}^k \sigma_j\partial_{\sigma_j}.
\end{equation}
It is easy to verify that $u$ is homogeneous of degree $\kappa$ if and only if
\begin{equation}\label{Zak}
\Za u = \kappa u.
\end{equation}
We note that 
\begin{align*}
 d(\delta_\lambda(z,t))&=\lambda^{m+(\alpha+1)k}\ dz\ dt,
\end{align*}
which motivates the definition of the \emph{homogeneous dimension} for the number
\begin{align}\label{Qa}
Q_\alpha=m+(\alpha+1)k.
\end{align}

We emphasise that \eqref{ra} is not a true gauge, because there is no underlying group structure associated to $\Ba$. We will respectively denote by
\begin{align}\label{BSr}
 B_r&=\{(z,\sigma)\in \R^N \mid \rho_\alpha(z,\sigma)<r\},\ \ \ \ \ S_r = \p B_r,
\end{align}
the ball and sphere centred at the origin with radius $r>0$. It is straightforward to verify that $\Ba$ is $\delta_\la$-homogeneous of degree two, i.e., 
\begin{equation}\label{deg2}
 \Ba(\delta_\lambda \circ u)=\lambda^2\delta_\lambda\circ (\Ba u).
\end{equation}

Functions which satisfy $\Ba u=0$ may not be necessarily smooth in the Euclidean sense. For instance, if $A = \frac{(\alpha+1)(2\alpha+m)}{k}$, then the function $P_\kappa(z,\sigma) =  |z|^{2(\alpha + 1)} - A |\sigma|^2,$ is a solution of $\Ba f = 0$, homogeneous of degree $\kappa = 2(\alpha+1)$.  It is thus necessary to introduce the following classes of ``smooth  functions''. We indicate with $x = (z,\sigma)\in \RN$ the generic point in $\RN$, and consider the vector fields $X_1,...,X_N$ defined by
\begin{align}\label{df}
& X_i(x)= \partial_{z_i},\ \ \  i=1, ...m,\ \ \ \  \ \ \ \ X_{m+j}(x)= \frac{|z|^{\alpha}}2 \partial_{\sigma_j},\ \ \   j=1, ...k,\ \ \alpha>0.
\end{align}
For an open set $\Om\subset \RN$ we define
\begin{align*}
 \Gamma^1(\Omega)=\{f\in C(\Omega)\mid f,X_jf\in C(\Omega)\},
\end{align*}
where the derivatives are meant in the distributional sense. We also set
\[
\Gamma^2(\Omega)=\{f\in C(\Omega)\mid f,X_jf\in \Gamma^1(\Omega)\}.
\] 
Thus, solutions to $\Ba u=0$ in $\Omega$ are taken to be of class $\Gamma^2(\Omega)$.

In \cite{Gjde}, the second named author proved that, with an explicitly given constant $C_\alpha>0$, 
 the function 
\begin{equation}\label{Ga}
 \Gamma_\alpha(z,\sigma)=\frac{C_{\alpha}}{\ra(z,\sigma)^{Q_{\alpha}-2}}
\end{equation}
is a fundamental solution for $-\Ba$ with singularity at $(0,0)$. Since, as we have observed above, the operator is invariant with respect to translations along $M = \{0\}\times \R^k$, one easily recognises that  \eqref{Ga} provides a fundamental solution at every point of this subspace of $\R^N$.
For $u, v\in \Gamma^1(\R^{N})$, we define the $\alpha$-gradient of $u$ to be
\begin{equation*}
 \nabla_\alpha u=\nabla_z u+\frac{|z|^\alpha}{2}\nabla_\sigma u,
\end{equation*}
and we set
\[
\sa\na u,\na v\da = \sa\nabla_z u,\nabla_z v\da + \frac{|z|^{2\alpha}}{4} \sa\nabla_\sigma u,\nabla_\sigma v\da.
\]
The square of the length of $\nabla_\alpha u$ is
\begin{equation*}
 |\nabla_\alpha u|^2=|\nabla_z u|^2+\frac{|z|^{2\alpha}}{4}|\nabla_\sigma u|^2.
\end{equation*}
The following lemma, collects the identities (2.12)-(2.14) in \cite{Gjde}.

\begin{lemma}\label{l:gradalpharho}
 Given a function $u$ one has in $\R^N\setminus\{0\}$,
 \begin{equation}\label{nara}
  \psi_{\alpha} \overset{def}{=} |\na \ra|^2 = \frac{|z|^{2\alpha}}{\ra^{2\alpha}},
 \end{equation}
and
 \begin{equation}\label{nara2}
 \sa\na u,\na \ra\da = \frac{\Za u}{\ra} \psi_\alpha.
 \end{equation}
 \end{lemma}
 The identity \eqref{nara2} is particularly important. We will repeatedly use the following basic consequence of it.

\begin{lemma}\label{L:tang}
For a given function $u$ we have
\[
|\nabla_\alpha u|^2 \ge \left(\frac{Z_\alpha u}{\rho_\alpha}\right)^2 \psi_\alpha.
\]
\end{lemma}

\begin{proof}
If $x \in M$ is a point at which $|\nabla_\alpha \rho_\alpha|^2 = \psi_\alpha = 0$ (see \eqref{nara}), there is nothing to prove. We can thus assume $\psi_\alpha(x) \not= 0$. Consider the unit vector at $x$
\[
\nu_\alpha = \frac{\nabla_\alpha \rho_\alpha}{|\nabla_\alpha \rho_\alpha|},
\]
and define 
\[
\delta_\alpha u = \nabla_\alpha u - \sa\nabla_\alpha u,\nu_\alpha\da \nu_\alpha.
\]
We have
\[
0\le |\delta_\alpha u|^2 = |\nabla_\alpha u|^2 - \sa\nabla_\alpha u,\nu_\alpha\da^2.
\]
The desired conclusion now follows by observing that \eqref{nara2} gives
\[
\sa\nabla_\alpha u,\nu_\alpha\da^2 = \left(\frac{Z_\alpha u}{\rho_\alpha}\right)^2 \psi_\alpha.
\]

\end{proof}

We close this section with an auxiliary result that will be needed in the proof of Theorem \ref{main1}. In what follows, for some fixed $R_1>R_0$, we let $\Om_1 = \{x\in \RN\mid \rho_\alpha(x)>R_1\}$.

\begin{proposition}\label{ms}
Let $f$ be a solution to \eqref{helmo0} in $\Omega_0=\{x\in \RN\mid \rho_\alpha(x)>R_0\}$ and suppose that $f$ satisfy \eqref{lp} for some $p>0$. Then $f \in L^{\infty}(\Omega_1)$, and there exists $C = C(m,k, \alpha,  p, R_0, R_1)>0$ such that
\begin{equation}\label{infty}
||f||_{L^\infty(\Om_1)} \le C \left(\int_{\Om_0} |f|^p dx\right)^{\frac 1p}.
\end{equation}
\end{proposition}

\begin{proof}
Given a point $(z_0, \sigma_0) \in \Om_1$, there are two possibilities: (i) $|z_0| < 2R_0$; (ii) $|z_0| \geq 2R_0$.
In the case (i), the estimate
\begin{equation}\label{est1}
|f(z_0, \sigma_0)| \leq C \left(\int_{\Om_0} |f|^p dx\right)^{1/p},
\end{equation}
is seen as follows.  First, for $y \in \mathbb R$, we define $g(z, \sigma, y)= e^{y} f(z, \sigma)$. It is clear that the function $g$ solves the following equation in $\Om_0\times \R_y\subset \R^{N+1}$,
\begin{equation}\label{g}
L  g \overset{def}{=} \Delta_z g +\p_{yy} g + \frac{|z|^{2\alpha}}{4}\Delta_{\sigma} g =0.
\end{equation}
Then by applying the arguments in the proof of Lemmas 1 \& 2, on p. 230-231 in \cite{W}, we obtain the following $L^2$ estimate when $|z_0| < 2R_0$
\begin{equation}\label{es2}
|f(z_0, \sigma_0)| \leq C(m, k, R_0, R_1) \left(\int_{C_r(z_0,\sigma_0)} |f|^2 dx\right)^{1/2},
\end{equation}
where we have let $C_r(z_0,\sigma_0)=\{(z,\sigma)\in \RN \mid |z-z_0|< r, |\sigma-\sigma_0| < r^{\alpha+1} \}$. Here, $r>0$ is sufficiently small, depending on $R_0, R_1$, so that $\overline C_r(z_0,\sigma_0) \subset \Omega_0$. 
The desired estimate \eqref{est1} follows  by H\"older inequality from \eqref{es2} when $p>2$. When $p \in (0,2)$ we can adapt the standard interpolation argument in \cite[Chapter 4]{HL} to reach the desired conclusion starting from \eqref{es2}.

In case (ii), when $|z_0| \geq 2R_0$, we consider the following rescaled function
\begin{equation}\label{tildef}
\tilde f(z,\sigma) = f(z, \sigma_0 + |z_0|^{\alpha} \sigma).
\end{equation}
Such $\tilde f$ solves the equation
\begin{equation}\label{res1}
\tilde L \tilde f \overset{def}= \Delta_z \tilde f + \frac{|z|^{2\alpha}}{4 |z_0|^{2\alpha}} \Delta_\sigma \tilde f = - \tilde f.
\end{equation}
One can easily check that, in the region $A= \{(z,\sigma) \mid |z-z_0| < R_0, |\sigma|< 1/2\}$, $\tilde L$ is uniformly elliptic, with ellipticity bounds which are universal ($2^{-2\alpha-2}$ from below, and $\max \left(1, \frac{3^{2\alpha}}{2^{2\alpha+2}} \right)$ from above will do). Moreover, since $|z_0| \geq 2R_0$, it follows from the  triangle inequality that $|z|> R_0$ for all $(z, \sigma) \in A$. Therefore, we have that for $(z, \sigma) \in A$,  $(z, \sigma_0 + |z_0|^{\alpha} \sigma) \in \Omega_0$.  By applying the Moser type estimate to $\tilde f$, see \cite[Theorem 4.1]{HL}, we find
\begin{equation}\label{tr1}
|f(z_0, \sigma_0)| = |\tilde f(z_0, 0)| \leq C \left( \int_{A} |\tilde f|^p \right)^{1/p} \leq C |z_0|^{-\frac{k\alpha}{p}} \left( \int_{\Om_0} |f|^p dx \right)^{1/p},
\end{equation}
where in the last inequality in \eqref{tr1} we have used the change of variable formula. Since we are in the case (ii), we have $|z_0|^{-\frac{k\alpha}{p}} \leq (2R_0)^{-\frac{k\alpha}{p}}$, and it thus follows from \eqref{tr1} that the estimate \eqref{est1} holds in this case as well, for a new constant $C=C(m,k, \alpha, p, R_0, R_1)$.
This concludes the proof.

\end{proof}

\section{Background results}\label{S:ri}

In this section to streamline the presentation, and also for the sake of the reader's understanding, we summarise some results from \cite{GS} that will play a key role in the present paper. Henceforth, to simplify the notation we will drop the subscript $\alpha$ in some of the above definitions. Specifically, we will write $Z$ instead of $Z_\alpha$ in \eqref{Za}, $Q$ instead of $Q_\alpha$ in \eqref{Qa}, $\rho$ instead of $\rho_\alpha$ in \eqref{ra}, $\psi$ instead of $\psi_\alpha$ in \eqref{nara}.  Let $f\in \Gamma^2(\Omega_0)$ be a solution to \eqref{kappa} in $\Om_0 = \{x\in \RN\mid \rho(x)>R_0\}$. We note from \eqref{deg2} that, by considering the rescaled function $g(x) = f(\delta_{\kappa^{-1}}(x))$ (this of course will rescale the radius $R_0$ as well, but this is of no influence in the proof of Theorem \ref{sharprel}), we can assume without restriction that $\kappa = 1$, hence $f$ solves 
\begin{equation}\label{helmo}
\Ba f = - f,\ \ \ \ \ \text{in}\ \Om_0.
\end{equation}  
The assumption $f\not\equiv 0$ in Theorem \ref{sharprel}, implies the existence of a radius $r_0>R_0$ such that
\begin{equation}\label{rzero}
\int_{S_{r_0}} f^2 \qui >0.
\end{equation} 
Otherwise, by the Federer coarea formula we would have
\[
\int_{\Om_0} f^2 \psi dx = \int_{R_0}^\infty \int_{S_t} f^2 \qui dt = 0,
\]
which of course implies $f\equiv 0$ in $\Om_0$.
With $f$ fixed as in \eqref{helmo}, for every $\ell\in \mathbb N$ we indicate by $u\in \Gamma^2(\Om_0)$ the function
\begin{equation}\label{uf}
u = \rho^\ell f.
\end{equation}
We observe from \eqref{uf} and \eqref{nara2} that 
\begin{equation}\label{zip}
\left(\frac{Z u}{\rho}\right)^2 \psi - |\na
u|^2 = \rho^{2\ell} \left[\left(\frac{Z f}{\rho}\right)^2 \psi - |\na f|^2\right].
\end{equation}
Next, for every $R_0<r<R$, we consider the ring domain $\Om=\{x\in \RN\mid r<\rho(x)<R\}$.   
For the function $u$ defined by \eqref{uf}, and for every $s\in \R$, $s\not=0$, we have the following basic Rellich identity established at the end of p. 708 in \cite{GS}
\begin{align}\label{gud}
& 2
\int_{\p \Om} \rho^s \left(\frac{Z u}{\rho}\right)^2 \sa Z,\nu\da \psi d\sigma  - \int_{\p\Om} \rho^s |\nabla_\alpha
u|^2 \langle Z,\nu\rangle d\sigma
\\
& + \int_{\p \Om} \rho^s u^2 \sa Z,\nu\da d\sigma
 +  \ell(\ell-Q+2)\int_{\p \Om} \rho^{s-2} u^2 \sa Z,\nu\da \psi d\sigma
\notag
\\
& = (2-Q-s)\int_\Om \rho^s |\nabla_\alpha u|^2 dx + 2(2\ell+ s) \int_\Om \rho^{s} \left(\frac{Zu}{\rho}\right)^2 \psi dx 
\notag\\
& + (Q+s) \int_\Om \rho^s u^2 dx
 +  \ell(\ell-Q+2)(Q+s-2) \int_\Om \rho^{s-2} u^2 \psi dx.
\notag
\end{align} 
The identity \eqref{gud} led in \cite{GS} to the introduction for $t>R_0$ of the following functional  
\begin{align}\label{Fell}
F(\ell,t) & = 2 \int_{S_t} \left(\frac{Z u}{\rho}\right)^2 \psi \frac{d\sigma}{|\nabla \rho|} - \int_{S_t} |\nabla_\alpha
u|^2  \quo +\int_{S_t} u^2  \quo
\\
& - \int_{S_t} \frac{u^2}{\rho} \qui  +  \ell(\ell-Q+2)\int_{S_t}  \left(\frac{u}{\rho}\right)^2  \qui\bigg\}.
\notag
\end{align}
With $r_0> R_0$ as in \eqref{rzero}, we define the number
\begin{equation}\label{Cfrzero}
C(f,r_0) \overset{def}{=} \frac{r_0^2 \left[\int_{S_{r_0}} |\nabla_\alpha
f|^2 \quo - \int_{S_{r_0}} \left(\frac{Z f}{\rho}\right)^2\qui\right]}{\int_{S_{r_0}}  f^2  \qui} + r_0.
\end{equation}
The reader should observe that, in view of 
Lemma \ref{L:tang}, we have $C(f,r_0)>0$. We will need the following result, which in \cite[Lemma 3.6]{GS} is formulated in a slightly different way, and in a different chronological order.

\begin{lemma}\label{L:1}
Let $r_0>R_0$ be as in \eqref{rzero}, and let $\ell_0 > \max\{Q-1, C(f,r_0)\}$ be fixed. Then for every $\ell \ge \ell_0$, we have
\begin{equation}\label{GS1}
F(\ell,r_0) >0.
\end{equation}
\end{lemma}

\begin{proof}
To prove \eqref{GS1}, observe that from definition \eqref{Fell} (using also that $\psi \leq 1$) we have the trivial estimate
\begin{align*}
F(\ell,r_0) & \ge \int_{S_{r_0}} \left(\frac{Z u}{\rho}\right)^2 \psi \frac{d\sigma}{|\nabla \rho|} - \int_{S_{r_0}} |\nabla_\alpha
u|^2  \quo
\\
& + \left\{\frac{\ell(\ell-Q+2) }{r_0^2} - \frac{1}{r_0}\right\}\int_{S_{r_0}}  u^2  \qui\bigg\}.
\end{align*}
From \eqref{uf}, \eqref{zip} we thus obtain 
\begin{align*}
F(\ell,r_0) & \ge r_0^{2\ell} \bigg[\int_{S_{r_0}} \left(\frac{Z f}{\rho}\right)^2\qui - \int_{S_{r_0}} |\nabla_\alpha
f|^2 \quo 
 + \left\{\frac{\ell(\ell-Q+2) }{r_0^2} - \frac{1}{r_0}\right\}\int_{S_{r_0}}  f^2  \qui\bigg].
\end{align*}
It is thus clear that we can accomplish $F(\ell,r_0)>0$ for any $\ell\ge \ell_0$, provided that the quantity between square brackets in the right-hand side is $>0$. If now $\ell_0 \ge Q-1$, then this will be the case  if
\[
\ell_0(\ell_0-Q+2) \ge \ell_0 >  \frac{r_0^2 \left[\int_{S_{r_0}} |\nabla_\alpha
f|^2 \quo - \int_{S_{r_0}} \left(\frac{Z f}{\rho}\right)^2\qui\right]}{\int_{S_{r_0}}  f^2  \qui} + r_0 = C(f,r_0),
\]
where in the last equality we have used \eqref{Cfrzero}.

\end{proof}

The following basic monotonicity result is \cite[Lemma 3.5]{GS}.

\begin{lemma}\label{L:Fmono}
If $\ell\ge \frac{Q-1}2$,
then for $R_0<r<R$ the following monotonicity property holds
\begin{equation}\label{Fsell}
r^{3-Q} F(\ell,r)\ \le\ R^{3-Q} F(\ell,R).
\end{equation} 
\end{lemma}

\section{Proof of Theorem \ref{sharprel}}\label{S:main}

In this section we prove Theorem \ref{sharprel}. We begin with a key  monotonicity property of a functional different from $F(\ell,t)$ in \eqref{Fell} above. Such result is one the main novelties in the proof of the improved estimate in Theorem \ref{sharprel}.

\begin{lemma}\label{L:ag}
Let $f$ solve \eqref{helmo}, and define $w = \rho^{\frac{Q-1}2} f$. If $Q\ge 3$, for every $t>R_0$ we consider the functional
\begin{align}\label{Gud}
G(w,t) & = 2 \int_{S_t} \left(\frac{Zw}{\rho}\right)^2 \qui - \int_{S_t} |\nabla_\alpha w|^2 \quo + \int_{S_t} w^2 \frac{d\sigma}{|\nabla \rho|}
\\
& - \frac{(Q-1)(Q-3)}4 \int_{S_t}  \left(\frac{w}{\rho}\right)^2 \qui.
\notag
\end{align} 
Then the following monotonicity property holds for $r_0<r<R$,
\[
\frac{G(w,r)}{r^{Q-1}} \le \frac{G(w,R)}{R^{Q-1}}.
\]
\end{lemma}

\begin{proof}
In the Rellich identity \eqref{gud}, we choose
\begin{equation}\label{ag}
s = - Q,\ \ \ \ \ \ \&\ \ \ \ \ell = \frac{Q-1}2.
\end{equation}
The right-hand side of \eqref{gud} becomes
\begin{equation}\label{pos}
2 \int_\Om \rho^s |\nabla_\alpha w|^2 dx - 2 \int_\Om \rho^{s} \left(\frac{Zw}{\rho}\right)^2 \psi dx 
 +  \frac{(Q-1)(Q-3)}2 \int_\Om \rho^{s-2} w^2 \psi dx\ \ge \ 0
\end{equation}
where in the last inequality we have used Lemma \ref{L:tang}, and the assumption $Q\ge 3$. Since $\Om=\{x\in \Rn\mid r<\rho(x)<R\}$, keeping the definition \eqref{Gud} in mind, we see that the left-hand side of \eqref{gud} becomes instead
\begin{equation}\label{gudd}
\frac{G(w,R)}{R^{Q-1}} - \frac{G(w,r)}{r^{Q-1}}.
\end{equation}
Combining \eqref{pos} with \eqref{gudd}, we reach the desired conclusion.

\end{proof}

Next, we prove that by replacing in \cite[Lemma 3.9]{GS} the functional $G(\rho)$ defined in equation (3.8) on p. 711, with $G(w,t)$ defined in \eqref{Gud} above, the conclusion continues to be valid unchanged under the weaker assumption \eqref{franznot} (such assumption was $\underset{R\to \infty}{\liminf}\ \int_{R<\rho(x)<2R} |f|^2 dx = 0$ in \cite{GS}). 

\begin{lemma}\label{L:gsimpr}
Let $f$ solve \eqref{helmo}, and define $w = \rho^{\frac{Q-1}2} f$, and $G(w,t)$ as in \eqref{Gud}. If $f$ satisfies 
\begin{equation}\label{franznot} 
\underset{R\to \infty}{\liminf}\ \frac{1}{R} \int_{R<\rho(x)<2R} |f|^2 dx = 0,
\end{equation}
then there exists a sequence $t_j\nearrow \infty$ such that $G(w,t_j)>0$.
\end{lemma}

\begin{proof}
Let $\ell_0$ be fixed as in Lemma \ref{L:1}, so that \eqref{GS1} hold for $t=r_0$. 
From \eqref{Fell}, where we take $u = \rho^{\ell_0} f$, and using \eqref{zip} with $\ell = \ell_0$, we obtain
\begin{align*}
F(\ell_0,t) & = \int_{S_{t}} \left(\frac{Z u}{\rho}\right)^2 \psi \frac{d\sigma}{|\nabla \rho|} 
\\
& t^{2\ell_0 - Q + 1} \bigg\{\int_{S_{t}} \left(\frac{Z w}{\rho}\right)^2 \psi \frac{d\sigma}{|\nabla \rho|} - \int_{S_{t}} |\nabla_\alpha
w|^2  \quo +\int_{S_{t}} w^2  \quo\bigg\}
\\
& + t^{2\ell_0}\bigg\{\ell_0(\ell_0-Q+2)\int_{S_{t}}  \left(\frac{f}{\rho}\right)^2 \qui - \int_{S_{t}} \frac{f^2}{\rho} \qui\bigg\}.
\notag
\end{align*}
If we let $\ell_1 = \frac{2\ell_0 - Q+1}2$, and keeping in mind the definition \eqref{Gud} above of the functional $G(w,t)$, proceeding similarly to the proof of \cite[Lemma 3.9]{GS}, we presently obtain
\begin{align}\label{Felluccinina}
F(\ell_0,t) & = t^{2\ell_1} G(w,t)   + \ell_1 t^{2(\ell_1-1)}\int_{S_t} Z(w^2) \qui
\\
& -  t^{2\ell_0-1}\left\{1 - \frac{\ell_0(\ell_0-n+2) +\ell_1^2 +4^{-1}(Q-1)(Q-3)}{t}\right\}\int_{S_t}  f^2  \qui
\notag\\
& \le t^{2\ell_1} G(w,t)   + \ell_1 t^{2(\ell_1-1)}\int_{S_t} Z(w^2) \qui,
\notag
\end{align}
provided that $t\ge r_1$, with $r_1>r_0$ sufficiently large.
We now show that  \eqref{franznot} implies the existence of a sequence $t_j\nearrow \infty$ such that 
\begin{equation}\label{Gtinfty}
\int_{S_{t_j}} Z(w^2)\qui\ \le\ 0.
\end{equation}
To prove such implication we argue by contradiction and assume there exists $R_1>R_0$ large such that  
\[
\int_{S_t} Z(w^2)\qui\ >\ 0\ \ \ \ \ \ \ \forall\ t\ge R_1.
\]
With the equation $w = \rho^{\frac{Q-1}2} f$ in mind, the divergence theorem gives for $R>r\ge R_1$ and $\Om=\{x\in \Rn\mid r<\rho(x)<R\}$,
\begin{align*}
& \int_{\p \Om} \frac{w^2}{\rho^Q} \sa Z,\nu\da \psi d\sigma = \int_{\Om} \operatorname{div}(\psi  \rho^{-Q} w^2 Z) dx = Q \int_\Om \psi  \rho^{-Q} w^2 dx + \int_\Om Z(\psi  \rho^{-Q} w^2) dx
\\
& = \int_\Om \psi  \rho^{-Q} Z(w^2) dx = \int_r^R t^{-Q} \int_{S_t} Z(w^2)  \qui dt\ >\ 0,
\end{align*}
where we have used that $Z(\rho^{-Q}) = - Q \rho^{-Q}$ and the coarea formula. Keeping in mind that $\sa Z,\nu\da = \frac{Z\rho}{|\nabla \rho|} = \frac{\rho}{|\nabla \rho|}$ on $\p \Om$, we thus obtain for $R\ge 2R_1$
\begin{align*}
& \int_{S_R} |f|^2 \qui = \int_{S_{R_1}} |f|^2 \qui + \int_{R_1}^R t^{-Q} \int_{S_t} Z(w^2)  \qui dt
\\
&  \ge \int_{R_1}^{2R_1} t^{-Q} \int_{S_t} Z(w^2)  \qui dt = C(f)>0.
\end{align*}
By the coarea formula again, and the fact that $\psi \le 1$, this implies for $R\ge 2R_1$
\begin{align*}
& \int_{R<\rho(x)<2R} |f|^2 dx\ \ge\  \int_{R<\rho(x)<2R} |f|^2 \psi dx
 = \int_R^{2R} \int_{S_t} |f|^2 \qui dt \ge C(f) \ R,
\end{align*}
which shows that it must be
\[
\underset{R\to \infty}{\liminf}\ \frac{1}{R} \int_{R<\rho(x)<2R} |f|^2 dx\ >\  0.
\]
We have thus proved that \eqref{franznot} implies \eqref{Gtinfty}.

At this point, we use the fact that, in Lemma \ref{L:1}, we have chosen $\ell_0 > Q-1$, and therefore for such $\ell_0$ we can apply Lemma \ref{L:Fmono}, obtaining for $t\ge r_1$
\begin{equation}\label{Fsellino}
0\ <\ r_0^{3-Q} F(\ell_0,r_0)\ \le\ t^{3-Q} F(\ell_0,t).
\end{equation} 
The estimate \eqref{Fsellino}, combined with \eqref{Felluccinina} and \eqref{Gtinfty}, implies the desired conclusion.

\end{proof}

We can now establish the following improved version of \cite[Lemma 3.11]{GS}.

\begin{lemma}\label{L:gsimpr2}
Under the assumptions of Lemma \ref{L:gsimpr}, suppose furthermore that $Q\ge 3$. There exists $C>0$ and $R_2>R_0$ sufficiently large such that for every $t>R_2$ one has
\[
\frac{G(w,t)}{t^{Q-1}} \ge C.
\]
\end{lemma}

\begin{proof}
By Lemma \ref{L:ag} we know that for every $R>r>r_0$ we have
\[
\frac{G(w,r)}{r^{Q-1}} \le \frac{G(w,R)}{R^{Q-1}}.
\]
If we now take $r = t_j$ sufficiently large, where $t_j\nearrow \infty$ is the sequence in Lemma \ref{L:gsimpr}, we reach the desired conclusion.

\end{proof}

Finally, with Lemma \ref{L:gsimpr2} in hands, we are able to obtain the following key improvement of \cite[Lemma 3.13]{GS}.

\begin{lemma}\label{L:gsimpr3}
Under the assumptions of Lemma \ref{L:gsimpr}, suppose furthermore that $Q\ge 3$. There exists $C>0$ and $R_3>R_0$ sufficiently large such that for every $t>R_3$ one has
\[
\int_{S_{t}} \left(\frac{Z f}{\rho}\right)^2 \qui  + \int_{S_{t}} f^2 \frac{d\sigma}{|\nabla \rho|}\ \ge\ C.
\]
\end{lemma}

\begin{proof}
Keeping in mind that $w = \rho^{\frac{Q-1}2} f$, and arguing as in the proof of \cite[Lemma 3.13]{GS}, we obtain for $t$ large enough
\begin{equation}\label{yummy}
\int_{S_{t}} \left(\frac{Z f}{\rho}\right)^2 \qui  + \int_{S_{t}} f^2 \frac{d\sigma}{|\nabla \rho|}\ge \frac 12 t^{1-Q} \left\{\int_{S_{t}} \left(\frac{Z w}{\rho}\right)^2 \qui  + \int_{S_{t}} w^2 \frac{d\sigma}{|\nabla \rho|}\right\}.
\end{equation}
Next, we have from \eqref{Gud},
\begin{align*}
& \int_{S_{t}} \left(\frac{Z w}{\rho}\right)^2 \qui  + \int_{S_{t}} w^2 \frac{d\sigma}{|\nabla \rho|} = G(w,t) + \int_{S_{t}} |\nabla_\alpha w|^2 \quo
- \int_{S_{t}} \left(\frac{Z w}{\rho}\right)^2 \qui
\\
& + \frac{(Q-1)(Q-3)}4 \int_{S_t}  \left(\frac{w}{\rho}\right)^2 \qui\ \ge\ G(w,t),
\end{align*}
since by Lemma \ref{L:tang} we  have
\[
|\nabla_\alpha w|^2\ge \left(\frac{Z w}{\rho}\right)^2 \psi.
\]
The latter estimate and Lemma \ref{L:gsimpr2}, finally give for large $t$
\[
\int_{S_{t}} \left(\frac{Z w}{\rho}\right)^2 \qui  + \int_{S_{t}} w^2 \frac{d\sigma}{|\nabla \rho|}\ \ge\ C t^{Q-1}.
\]
Substituting this estimate in \eqref{yummy}, we reach the desired conclusion.

\end{proof}

We can finally give the
\begin{proof}[Proof of Theorem \ref{sharprel}]
Once the improved estimate in Lemma \ref{L:gsimpr3} is available, the sought for conclusion \eqref{franzosharp} follows    by  a straightforward modification  of  the proof of the Main Theorem in \cite[p. 714-715]{GS}. We nevertheless provide the relevant details for the sake of completeness. To prove \eqref{franzosharp} it suffices to show that
\begin{equation}\label{fr1} 
\underset{R\to \infty}{\liminf}\ \frac{1}{R} \int_{R<\rho(x)<2R} |f|^2 dx\ >\  0.
\end{equation}
Suppose that \eqref{fr1} be not true. Then by combining Lemma \ref{L:gsimpr3}, Lemma \ref{L:tang}  and Federer's coarea formula, we have for some $C>0$ and all $R$ sufficiently large,
\begin{equation}\label{ty1}
\int_{\frac{5}{4} R < \rho(x) < \frac{7}{4}R} \left(|\nabla_\alpha f|^2 + |f|^2 \right) dx \geq CR.
\end{equation}
The energy inequality on p. 715 in \cite{GS} gives
\begin{equation}\label{e}
\int_{\frac{5}{4} R < \rho(x) < \frac{7}{4}R} |\nabla_{\alpha} f|^2 dx \leq C_1 \int_{R< \rho(x) < 2R} |f|^2 dx,
\end{equation} for some $C_1>0$ universal. This together with \eqref{ty1} gives for some $C_2>0$ and all large $R$,
\begin{equation}\label{hy1}
\int_{R< \rho(x) < 2R} |f|^2 dx \geq C_2R,
\end{equation}
which is a contradiction. We conclude that \eqref{fr1} must be true, thus completing the proof of the theorem.

\end{proof}

\begin{proof}[Proof of Theorem \ref{main1}]
We first assume that  $2\le p \le \frac{2Q}{Q-1}$. Applying H\"older inequality to \eqref{franzosharp}, we obtain
\begin{equation}\label{rellich2}
M R \le C(Q, p) \left(\int_{R<\rho(x)<2R} |f|^p dx\right)^{\frac 2p} R^{Q(1-\frac 2p)}.
\end{equation}
When $p = \frac{2Q}{Q-1}$, we have $Q(1-\frac 2p)=1$, and \eqref{rellich2} gives
\[
0<M  \le C(n,p) \left(\int_{R<\rho(x)<2R} |f|^p dx\right)^{\frac 2p}.
\]
However, this inequality is contradictory with the assumption \eqref{lp} for such $p$, since by letting $R\to\infty$, the right-hand side converges to $0$ by Lebesgue dominated convergence.

We now analyse the case $0<p<2$. In such range, under the hypothesis $f\not\equiv 0$ in $\Om_0$, the sharp Rellich type  inequality \eqref{franzosharp} trivially implies for $R>R_1>>R_0$, 
\begin{equation}\label{rellich3}
M R \le ||f||^{2-p}_{L^\infty(\Om_1)} \int_{R<\rho(x)<2R} |f|^p dx.
\end{equation}
By the unique continuation property of the Helmholtz equation \eqref{helmo}, see \cite{Gjde}, we must have $f\not\equiv 0$ in $\Om_1$, and therefore $||f||_{L^\infty(\Om_1)}>0$. On the other hand, if $f$ satisfies \eqref{lp}, then by Proposition \ref{ms} we also have  $||f||_{L^\infty(\Om_1)} < \infty$. 
Letting $R\to \infty$ in \eqref{rellich3}, we thus reach a contradiction again. 

\end{proof}






\bibliographystyle{amsplain}

\begin{thebibliography}{10}








\bibitem{A}
S. Agmon, \emph{Lower bounds for solutions of Schr\"odinger equations}.
J. Analyse Math. 23 (1970), 1-25.

\bibitem{BG}
A. Banerjee \& N. Garofalo, \emph{An observation on eigenfunctions of the Laplacian}, preprint, 2023. ArXiv  2311.05905


\bibitem{B}
S. M. Baouendi, \emph{Sur une classe d'op\'erateurs elliptiques d\'eg\'en\'er\'es} (French). Bull. Soc. Math. France, 95~1967, 45-87.

\bibitem{BFIh}
W. Bauer, K. Furutani \& C. Iwasaki, \emph{Spectral analysis and geometry of sub-Laplacian and related Grushin-type operators}. Partial differential equations and spectral theory. Oper. Theory Adv. Appl., vol. 211, 183-290, Birkh\"{a}user/Springer Basel AG, Basel, 2011.

\bibitem{BFI}
W. Bauer, K. Furutani \& C. Iwasaki, \emph{Fundamental solution of a higher step Grushin type operator}. Adv. Math. 271~(2015), 188-234.

\bibitem{BGG1}
R. Beals, B. Gaveau \& P. Greiner, \emph{
On a geometric formula for the fundamental solution of subelliptic Laplacians}.
Math. Nachr. 181~(1996), 81-163.
  
\bibitem{BGG2}
R. Beals, B. Gaveau \& P. Greiner, \emph{Green's functions for some highly degenerate elliptic operators}. J. Funct. Anal. \textbf{165}~(1999), no. 2, 407-429.


\bibitem{Be}
W. Beckner, \emph{On the Grushin operator and hyperbolic symmetry}.
Proc. Amer. Math. Soc.129~(2001), no.4, 1233-1246.

\bibitem{CSS}
L. Caffarelli, S. Salsa \& L. Silvestre, \emph{
Regularity estimates for the solution and the free boundary of the obstacle problem for the fractional Laplacian}.
Invent. Math.171~(2008), no.2, 425-461.


\bibitem{CaS}
L. Caffarelli \& L. Silvestre, \emph{ An extension problem related to the fractional Laplacian}.
Comm. Partial Differential Equations 32 (2007), no. 7-9, 1245-1260.


\bibitem{Fo}
G. Folland, \emph{A fundamental solution for a subelliptic operator}, Bull. Amer. Math. Soc., 79~(1973), 373-376.

\bibitem{FLto}
B. Franchi \& E. Lanconelli, \emph{Une m\'etrique associ\'ee \`a une classe d'op\'erateurs elliptiques d\'eg\'en\'er\'es.
A metric associated with a class of degenerate elliptic operators}.
Rend. Sem. Mat. Univ. Politec. Torino (1983), 105-114.

\bibitem{FL}
B. Franchi \& E. Lanconelli, \emph{H\"older regularity theorem for a class of linear nonuniformly elliptic operators with measurable coefficients}. Ann. Scuola Norm. Sup. Pisa Cl. Sci. (4) 10~ (1983), no. 4, 523-541.

\bibitem{FL84}
B. Franchi \& E. Lanconelli, \emph{An embedding theorem for Sobolev spaces related to nonsmooth vector fields and Harnack inequality}.
Comm. Partial Differential Equations 9 (1984), no. 13, 1237-1264.

\bibitem{FL85}
B. Franchi \& E. Lanconelli, \emph{ Une condition g\'eom\'etrique pour l'in\'egalit\'e de Harnack.
A geometric condition for Harnack's inequality}.
J. Math. Pures Appl. (9) 64 (1985), no. 3, 237-256.


\bibitem{FS}
B. Franchi \& R. Serapioni, \emph{Pointwise estimates for a class of strongly degenerate
elliptic operators: A geometrical approach}, Ann. Sc. Norm. Sup. Pisa \textbf{14}~(1987), 527-568.


\bibitem{Gjde}
N. Garofalo, \emph{Unique continuation for a class of elliptic operators which degenerate on a manifold of arbitrary codimension}. J. Differential Equations, 104~(1993), no.\ 1, 117-146.

\bibitem{Gft}
N. Garofalo, \emph{Fractional thoughts}.
Contemp. Math., 723
American Mathematical Society, [Providence], RI, 2019, 1-135.


\bibitem{GR}
N. Garofalo \& K. Rotz, \emph{ Properties of a frequency of Almgren type for harmonic functions in Carnot groups}.
Calc. Var. Partial Differential Equations 54~ (2015), no. 2, 2197-2238.


\bibitem{GS}
N. Garofalo \& Z. Shen, \emph{Absence of positive eigenvalues for a class of subelliptic operators}.
Math. Ann. 304 (1996), no. 4, 701-715.

\bibitem{GTpotan}
N. Garofalo \& G. Tralli, \emph{Heat kernels for a class of hybrid evolution equations}, Potential Anal. 59~(2023), 823-856. 


\bibitem{Gre}
P. C. Greiner, \emph{A fundamental solution for a nonelliptic partial differential operator}.
Canadian J. Math. 31~(1979), no.5, 1107-1120.

\bibitem{Gr1}
V. V. Grushin, \emph{A certain class of hypoelliptic operators}, (Russian) Mat. Sb. (N.S.) \textbf{83} (125)~(1970), 456-473.

\bibitem{Gr2}
V. V. Grushin, \emph{A certain class of elliptic pseudodifferential operators that are degenerate on a submanifold}, (Russian) Mat. Sb. (N.S.) \textbf{84} (126)~(1971), 163-195.

\bibitem{GV1}
V. V. Grushin \& M. I. Vishik, \emph{A certain class of degenerate elliptic equations of higher orders}.
Mat. Sb. (N.S.) 79(121) (1969), 3-36.

\bibitem{GV2}
V. V. Grushin \& M. I. Vishik, \emph{Boundary value problems for elliptic equations which are degenerate on the boundary of the domain}.
Mat. Sb. (N.S.) 80(122) (1969), 455-491.




\bibitem{HL}
Q. Han \& F. Lin, \emph{Elliptic partial differential equations},
Courant Lecture Notes in Mathematics, 1. New York University, Courant Institute of Mathematical Sciences, New York; American Mathematical Society, Providence, RI, 1997. x+144 pp. 

\bibitem{Je}
D. S. Jerison, \emph{The Dirichlet problem for the Kohn Laplacian on the Heisenberg group. II}.
J. Functional Analysis 43~(1981), no.2, 224-257.

\bibitem{Kap}
A. Kaplan, \emph{Fundamental solutions for a class of hypoelliptic PDE generated
by composition of quadratic forms}, Trans. Amer. Math. Soc., 258, 1 ~(1980), 147-153.


\bibitem{Ka}
T. Kato, \emph{Growth properties of solutions of the reduced wave equation with a variable coefficient}.
Comm. Pure Appl. Math.12~(1959), 403-425.

\bibitem{MM}
R. Monti \& D. Morbidelli, \emph{Kelvin transform for Grushin operators and critical semilinear equations}. Duke Math. J., 131 (1)~(2006), 167-202. 


\bibitem{RS}
M. Reed \& B. Simon, \emph{Methods of modern mathematical physics}. IV. Analysis of operators
Academic Press [Harcourt Brace Jovanovich, Publishers], New York-London, 1978, xv+396 pp.


\bibitem{Rel}
F. Rellich, \emph{\"Uber das asymptotische Verhalten der L\"osungen von $\Delta u + \la u = 0$ in unendlichen Gebieten}. (German) Jber. Deutsch. Math.-Verein. 53~(1943), 57-65.

\bibitem{RS}
L. P. Rothschild \& E. M. Stein,
\emph{Hypoelliptic differential operators and nilpotent groups}, 
Acta Math. 137~(1976), no. 3-4, 247-320. 


\bibitem{W}
L. Wang, \emph{H\"older estimates for subelliptic operators}.  J. of Functional Analysis, 199~(2003), 228-242.





\end{thebibliography}

\end{document}